\documentclass{amsart}%
\usepackage{amsfonts}
\usepackage{amsmath}
\usepackage{amssymb}
\usepackage{graphicx}%
\providecommand{\U}[1]{\protect\rule{.1in}{.1in}}
\newtheorem{theorem}{Theorem}
\theoremstyle{plain}

\newtheorem{example}{Example}

\newtheorem{lemma}{Lemma}

\newtheorem{remark}{Remark}

\numberwithin{equation}{section}
\begin{document}
\title[Finite sums and continued fractions]{Transformation formulas of finite sums\\into continued fractions}
\author{Daniel Duverney, Takeshi Kurosawa and Iekata Shiokawa}
\address{\"{y} }
\email{\"{y} }
\date{June. 17, 2020}
\subjclass{ }

\begin{abstract}
We state and prove three general formulas allowing to transform formal finite
sums into formal continued fractions and apply them to generalize certain
expansions in continued fractions given by Hone and Varona.

\end{abstract}
\maketitle

\section{Introduction}

Let $n$ be a positive integer, and let $x_{1},$ $x_{2},$ $\ldots,$ $x_{n},$
$\ldots,$ $y_{1},$ $y_{2},$ $\ldots,$ $y_{n},$ $\ldots$ be indeterminates. We
define%
\begin{equation}
\sigma_{n}=\sum_{k=1}^{n}\frac{y_{k}}{x_{k}},\quad\quad\tau_{n}=\sum_{k=1}%
^{n}\left(  -1\right)  ^{k-1}\frac{y_{k}}{x_{k}}. \label{T1}%
\end{equation}
Then, $\sigma_{n}$ and $\tau_{n}$ are rational functions of the indeterminates
$x_{1},$ $x_{2},$ $\ldots,$ $x_{n},$ $y_{1},$ $y_{2},$ $\ldots,$ $y_{n}$ with
coefficients in the field $\mathbb{Q}.$ The purpose of this paper is to give
three formulas allowing to transform $\sigma_{n}$ and $\tau_{n}$ into
continued fractions of the form%
\[
R_{m}=\frac{a_{1}}{b_{1}}%
\genfrac{}{}{0pt}{}{{}}{+}%
\frac{a_{2}}{b_{2}}%
\genfrac{}{}{0pt}{}{{}}{+\cdots}%
\genfrac{}{}{0pt}{}{{}}{+}%
\dfrac{a_{m}}{b_{m}},
\]
where $m$ is an increasing function of $n$ and $a_{1},$ $a_{2},$ $\ldots,$
$a_{m},$ $b_{1},$ $b_{2},$ $\ldots,$ $b_{m}$ are rational functions of
$x_{1},$ $x_{2},$ $\ldots,$ $x_{n},$ $y_{1},$ $y_{2},$ $\ldots,$ $y_{n}$ with
coefficients in $\mathbb{Q}.$ These formulas are given by Theorems
\ref{ThEuler}, \ref{ThHone}, and \ref{ThVarona} below. \medskip

For every sequence $\left(  u_{k}\right)  _{k\geq1}$ of indeterminates, we
define $u_{0}=1$ and%
\begin{equation}
\theta u_{k}=\frac{u_{k+1}}{u_{k}},\quad\theta^{2}u_{k}=\theta\left(  \theta
u_{k}\right)  =\frac{u_{k+2}u_{k}}{u_{k+1}^{2}}\quad\left(  k\geq0\right)
.\label{fn}%
\end{equation}
By (\ref{fn}) we see at once that%
\begin{equation}
u_{k}.\theta u_{k}=u_{k+1},\quad\theta u_{k}.\theta^{2}u_{k}=\theta
u_{k+1}\quad\left(  k\geq0\right)  .\label{Rule}%
\end{equation}

\begin{theorem}
\label{ThEuler}For every positive integer $n,$%
\begin{equation}
\sum_{k=1}^{n}\left(  -1\right)  ^{k-1}\frac{y_{k}}{x_{k}}=\frac{a_{1}}{b_{1}}%
\genfrac{}{}{0pt}{}{{}}{+}%
\frac{a_{2}}{b_{2}}%
\genfrac{}{}{0pt}{}{{}}{+\cdots}%
\genfrac{}{}{0pt}{}{{}}{+}%
\dfrac{a_{n}}{b_{n}}, \label{CFEuler2}%
\end{equation}
where $a_{1}=y_{1},$ $b_{1}=x_{1},$ and
\begin{equation}
a_{k}=\theta y_{k-1}\theta x_{k-2},\quad b_{k}=\theta x_{k-1}-\theta
y_{k-1}\quad\left(  2\leq k\leq n\right)  . \label{CFEuler1}%
\end{equation}

\end{theorem}

Theorem \ref{ThEuler} is a mere rewording of Euler's well-known formula
\cite{Euler}%
\[
\frac{1}{A}-\frac{1}{B}+\frac{1}{C}-\frac{1}{D}+\cdots=\frac{1}{A}%
\genfrac{}{}{0pt}{}{{}}{+}%
\frac{A^{2}}{B-A}%
\genfrac{}{}{0pt}{}{{}}{+}%
\frac{B^{2}}{C-B}%
\genfrac{}{}{0pt}{}{{}}{+}%
\dfrac{C^{2}}{D-C}%
\genfrac{}{}{0pt}{}{{}}{+\cdots}%
.
\]
Hence Theorem \ref{ThEuler} is far from being new. However, it seems
interesting to state and prove it by using the operator $\theta.$

\begin{theorem}
\label{ThHone}For every integer $n\geq1,$%
\begin{equation}
\sum_{k=1}^{n}\frac{y_{k}}{x_{k}}=\frac{a_{1}}{b_{1}}%
\genfrac{}{}{0pt}{}{{}}{+}%
\frac{a_{2}}{b_{2}}%
\genfrac{}{}{0pt}{}{{}}{+\cdots}%
\genfrac{}{}{0pt}{}{{}}{+}%
\dfrac{a_{2n}}{b_{2n}}, \label{Hone3}%
\end{equation}
where $a_{1}=y_{1},$ $b_{1}=x_{1}-y_{1},$ and for $k\geq1$%
\begin{align}
a_{2k}  &  =\theta y_{k-1},\quad a_{2k+1}=\theta^{2}y_{k-1},\label{Hone41}\\
\quad b_{2k}  &  =x_{k-1},\quad b_{2k+1}=\frac{\theta^{2}x_{k-1}-\theta
^{2}y_{k-1}}{x_{k-1}}. \label{Hone42}%
\end{align}

\end{theorem}

Theorem \ref{ThHone} has been given by Hone in \cite{Hone} in the special case
where $y_{k}=1$ for every $k\geq1$ and $x_{k}$ is a sequence of positive
integers such that $x_{1}\geq2$ and $x_{k}$ divides $\theta^{2}x_{k}-1$ for
every $k\geq1.$ In this case, (\ref{Hone3}) leads to the expansion of the
infinite series $\sum_{k=1}^{+\infty}x_{k}^{-1}$ in regular continued fraction.

\begin{theorem}
\label{ThVarona}For every integer $n\geq2,$%
\begin{equation}
\sum_{k=1}^{n}\left(  -1\right)  ^{k-1}\frac{y_{k}}{x_{k}}=\frac{a_{1}}{b_{1}}%
\genfrac{}{}{0pt}{}{{}}{+}%
\frac{a_{2}}{b_{2}}%
\genfrac{}{}{0pt}{}{{}}{+\cdots}%
\genfrac{}{}{0pt}{}{{}}{+}%
\dfrac{a_{3n-4}}{b_{3n-4}}, \label{Varona4}%
\end{equation}
where
\begin{align}
a_{1}  &  =y_{1}^{2},\quad a_{2}=x_{1}y_{2},\quad a_{3}=\theta y_{2},\quad
a_{4}=x_{1},\label{Varona5}\\
b_{1}  &  =x_{1}y_{1},\quad b_{2}=\theta x_{1}-\theta y_{1},\quad b_{3}%
=\theta^{2}x_{1}-x_{1},\quad b_{4}=1, \label{Varona6}%
\end{align}
and for $k\geq2$%
\begin{align}
a_{3k-1}  &  =y_{k+1},\quad a_{3k}=y_{k}\theta^{2}y_{k},\quad a_{3k+1}%
=1,\label{Varona7}\\
\quad b_{3k-1}  &  =x_{k}y_{k}-y_{k+1},\quad b_{3k}=\frac{\theta^{2}%
x_{k}-\theta^{2}y_{k}}{x_{k}}-1,\quad b_{3k+1}=1 \label{Varona8}%
\end{align}

\end{theorem}

Theorem \ref{ThVarona} has been proved first by Varona \cite{Varona} in the
special case where $y_{k}=1$ for every $k\geq1$ and $x_{k}$ is a sequence of
positive integers satisfying the same conditions as in Theorem \ref{ThHone}.
In this case (\ref{Varona4}) leads to the expansion of the infinite series
$\sum_{k=1}^{+\infty}\left(  -1\right)  ^{k}x_{k}^{-1}$ in regular continued
fraction.\medskip

In Section \ref{sec:notaion}, we recall some basic fact on continued fractions
and prove transformation formulas of continued fractions into finite sums.
Theorems \ref{ThEuler} and \ref{ThHone} will be proved in Section
\ref{sec:proofEandH} and Theorem \ref{ThVarona} in Section \ref{sec:proofV}.
Finally, in Section \ref{sec:Hone} and \ref{sec:Varona} we will give examples
of applications of Theorems \ref{ThHone} and \ref{ThVarona} by generalizing
Hone and Varona expansions. Indeed, we will define the sequence $(x_{n})$ by
the recurrence relation
\[
x_{n+2}x_{n}=x_{n+1}^{2}(F_{n}(x_{n},x_{n+1})+1)\qquad(n\geq0)
\]
with the initial conditions $x_{0}=1$ and $x_{1}\in{\mathbb{Z}}_{>0}$, where
$F_{n}(X,Y)$ are nonzero polynomials with positive integer coefficients such
that $F_{n}(0,0)=0$ for all $n\geq0$. It turns out that $(x_{n})$ is a
sequence of positive integers such that $x_{n}~|~x_{n+1}$ and $x_{n}%
~|~F_{n}(x_{n},x_{n+1})$ for every $n\geq0$. For any positive integer $h$, we
define the series
\[
S=\sum_{n=0}^{\infty}\frac{h^{n}}{x_{n+1}}.
\]
Applying Theorem \ref{ThHone} with $y_{n}=h^{n}$ and letting $n\rightarrow
\infty$, we have
\begin{equation}
S=\frac{a_{1}}{b_{1}}%
\genfrac{}{}{0pt}{}{{}}{+}%
\frac{a_{2}}{b_{2}}%
\genfrac{}{}{0pt}{}{{}}{+\cdots}%
\genfrac{}{}{0pt}{}{{}}{+}%
\dfrac{a_{n}}{b_{n}}%
\genfrac{}{}{0pt}{}{{}}{+\cdots}%
, \label{CFS}%
\end{equation}
where $a_{1}=1,$ $b_{1}=x_{1}-h,$ and for $k\geq1$%
\begin{align*}
a_{2k}  &  =h,\quad a_{2k+1}=1,\\
\quad b_{2k}  &  =x_{k-1},\quad b_{2k+1}=\frac{\theta^{2}x_{k-1}-1}{x_{k-1}%
}=\frac{F_{k-1}\left(  x_{k-1},x_{k}\right)  }{x_{k-1}}.
\end{align*}
are rational integers. Similarly, using Theorem \ref{ThVarona}, we get in
(\ref{T}) the continued fraction expansion of the alternating series
\begin{equation}
T=\sum_{n=0}^{\infty}\left(  -1\right)  ^{n}\frac{h^{n}}{x_{n+1}}=\frac{a_{1}%
}{b_{1}}%
\genfrac{}{}{0pt}{}{{}}{+}%
\frac{a_{2}}{b_{2}}%
\genfrac{}{}{0pt}{}{{}}{+\cdots}%
\genfrac{}{}{0pt}{}{{}}{+}%
\dfrac{a_{n}}{b_{n}}%
\genfrac{}{}{0pt}{}{{}}{+\cdots}
\label{CFT}%
\end{equation}
where%
\begin{align*}
a_{1}  &  =1,\quad a_{2}=hx_{1},\quad a_{3}=h,\quad a_{4}=x_{1},\\
b_{1}  &  =x_{1},\quad b_{2}=\frac{x_{2}}{x_{1}}-h,\quad b_{3}=F_{1}\left(
x_{1},x_{2}\right)  +1-x_{1},\quad b_{4}=1,
\end{align*}
and for $k\geq2$%
\begin{align*}
a_{3k-1}  &  =h^{k},\quad a_{3k}=h^{k-1},\quad a_{3k+1}=1,\\
\quad b_{3k-1}  &  =h^{k-1}\left(  x_{k}-h\right)  ,\quad b_{3k}=\frac
{F_{k}\left(  x_{k},x_{k+1}\right)  }{x_{k}}-1,\quad b_{3k+1}=1.
\end{align*}

The simplest of all sequences $(x_{n})$ satisfies the recurrence relation
\[
x_{n+2}x_{n}=x_{n+1}^{2}(x_{n}+1)\qquad(n\geq0)
\]
In the case $x_{0}=x_{1}=1,$ $(x_{n})$ is sequence A001697 of the On-line
Encyclopedia of Integer Sequences, which also satisfies%
\[
x_{n+1}=x_{n}\left(  \sum_{k=0}^{n}x_{k}\right)  \quad\left(  n\geq0\right)
.
\]
Taking $h=1$ in (\ref{CFS}) and (\ref{CFT}), we find remarkable formulas:
\begin{equation}
\left[  1;1,x_{1},1,x_{2},1,x_{3},1,x_{4},\ldots,1,x_{k},\ldots\right]
=\sum_{n=1}^{\infty}\frac{1}{x_{n}}, \label{Nouv1}%
\end{equation}%
\begin{equation}
\left[  0;1,1,1,x_{1},x_{2},x_{3},x_{4},x_{5},\ldots,x_{n},\ldots\right]
=\sum_{n=1}^{\infty}\frac{(-1)^{n-1}}{x_{n}}. \label{Nouv2}%
\end{equation}
See Examples \ref{ex:1} and \ref{ex:3} below.

\section{Notations and lemmas}

\label{sec:notaion}

For every positive integer $n,$ we define polynomials $P_{n}$ and $Q_{n}$ of
the indeterminate $a_{1},$ $a_{2},$ $\ldots,$ $b_{1},$ $b_{2},$ $\ldots$ by
$P_{0}=0,$ $Q_{0}=1$ and%
\begin{equation}
\frac{a_{1}}{b_{1}}%
\genfrac{}{}{0pt}{}{{}}{+}%
\frac{a_{2}}{b_{2}}%
\genfrac{}{}{0pt}{}{{}}{+\cdots}%
\genfrac{}{}{0pt}{}{{}}{+}%
\dfrac{a_{n}}{b_{n}}=\frac{P_{n}}{Q_{n}}\quad\left(  n\geq1\right)  .
\label{CFEuler3}%
\end{equation}
Then, we have for every $k\geq0$%
\begin{equation}
\left\{
\begin{array}
[c]{c}%
P_{k+2}=b_{k+2}P_{k+1}+a_{k+2}P_{k}\\
Q_{k+2}=b_{k+2}Q_{k+1}+a_{k+2}Q_{k}%
\end{array}
\right.  \label{Rec}%
\end{equation}
and also%
\begin{equation}
P_{k+1}Q_{k}-P_{k}Q_{k+1}=\left(  -1\right)  ^{k}a_{1}a_{2}\cdots a_{k+1}%
\quad\left(  k\geq0\right)  . \label{Delta}%
\end{equation}
From (\ref{Delta}) one obtains immediately a well-known transformation formula
of continued fractions into a finite sum: for every $n\geq1,$
\begin{equation}
\frac{a_{1}}{b_{1}}%
\genfrac{}{}{0pt}{}{{}}{+}%
\frac{a_{2}}{b_{2}}%
\genfrac{}{}{0pt}{}{{}}{+\cdots}%
\genfrac{}{}{0pt}{}{{}}{+}%
\dfrac{a_{n}}{b_{n}}=\sum_{k=0}^{n-1}\left(  -1\right)  ^{k}\frac{a_{1}%
a_{2}\cdots a_{k+1}}{Q_{k+1}Q_{k}}. \label{SumEuler}%
\end{equation}
Two other transformation formulas of continued fractions into finite sums are
given by the following lemmas.

\begin{lemma}
\label{LemTransfHone}For every integer $n\geq1,$%
\begin{equation}
\frac{a_{1}}{b_{1}}%
\genfrac{}{}{0pt}{}{{}}{+}%
\frac{a_{2}}{b_{2}}%
\genfrac{}{}{0pt}{}{{}}{+\cdots}%
\genfrac{}{}{0pt}{}{{}}{+}%
\dfrac{a_{2n}}{b_{2n}}=\sum_{k=0}^{n-1}\frac{a_{1}a_{2}\cdots a_{2k+1}%
b_{2k+2}}{Q_{2k}Q_{2k+2}}. \label{SumHone}%
\end{equation}

\end{lemma}

\begin{proof}
Replacing $n$ by $2n$ in (\ref{SumEuler}), we obtain%
\begin{align*}
\frac{a_{1}}{b_{1}}%
\genfrac{}{}{0pt}{}{{}}{+}%
\frac{a_{2}}{b_{2}}%
\genfrac{}{}{0pt}{}{{}}{+\cdots}%
\genfrac{}{}{0pt}{}{{}}{+}%
\dfrac{a_{2n}}{b_{2n}}  &  =\sum_{m=0}^{2n-1}\left(  -1\right)  ^{m}%
\frac{a_{1}a_{2}\cdots a_{m+1}}{Q_{m+1}Q_{m}}\\
&  =\sum_{k=0}^{n-1}\left(  \frac{a_{1}a_{2}\cdots a_{2k+1}}{Q_{2k+1}Q_{2k}%
}-\frac{a_{1}a_{2}\cdots a_{2k+2}}{Q_{2k+2}Q_{2k+1}}\right) \\
&  =\sum_{k=0}^{n-1}a_{1}a_{2}\cdots a_{2k+1}\frac{Q_{2k+2}-a_{2k+2}Q_{2k}%
}{Q_{2k+2}Q_{2k+1}Q_{2k}},
\end{align*}
which yields (\ref{SumHone}) since $Q_{2k+2}=b_{2k+2}Q_{2k+1}+a_{2k+2}Q_{2k}$
for every $k\geq0.$
\end{proof}

\begin{lemma}
\label{LemTransfVarona}For every integer $n\geq1,$%
\begin{align}
\frac{a_{1}}{b_{1}}%
\genfrac{}{}{0pt}{}{{}}{+}%
\frac{a_{2}}{b_{2}}%
\genfrac{}{}{0pt}{}{{}}{+\cdots}%
\genfrac{}{}{0pt}{}{{}}{+}%
\dfrac{a_{3n-1}}{b_{3n-1}}  &  =\frac{a_{1}}{Q_{1}}-\frac{a_{1}a_{2}}%
{Q_{1}Q_{2}}\label{SumVarona}\\
&  \qquad+\sum_{k=1}^{n-1}\left(  -1\right)  ^{k-1}a_{1}a_{2}\cdots
a_{3k}\frac{b_{3k+1}b_{3k+2}+a_{3k+2}}{Q_{3k-1}Q_{3k+2}}.\nonumber
\end{align}

\end{lemma}

\begin{proof}
We know by (\ref{SumEuler}) that%
\begin{multline*}
\frac{P_{3n-1}}{Q_{3n-1}}=\frac{a_{1}}{b_{1}}%
\genfrac{}{}{0pt}{}{{}}{+}%
\frac{a_{2}}{b_{2}}%
\genfrac{}{}{0pt}{}{{}}{+\cdots}%
\genfrac{}{}{0pt}{}{{}}{+}%
\dfrac{a_{3n-1}}{b_{3n-1}}=\frac{a_{1}}{Q_{1}}-\frac{a_{1}a_{2}}{Q_{1}Q_{2}%
}+\sum_{m=2}^{3n-2}\left(  -1\right)  ^{m}\frac{a_{1}a_{2}\cdots a_{m+1}%
}{Q_{m}Q_{m+1}}\\
=\frac{a_{1}}{Q_{1}}-\frac{a_{1}a_{2}}{Q_{1}Q_{2}}+\sum_{k=1}^{n-1}\left(
-1\right)  ^{k-1}a_{1}a_{2}\cdots a_{3k}\left(  \frac{1}{Q_{3k-1}Q_{3k}}%
-\frac{a_{3k+1}}{Q_{3k}Q_{3k+1}}+\frac{a_{3k+1}a_{3k+2}}{Q_{3k+1}Q_{3k+2}%
}\right)  .
\end{multline*}
Since $Q_{3k+1}-a_{3k+1}Q_{3k-1}=b_{3k+1}Q_{3k},$ we obtain%
\[
\frac{P_{3n-1}}{Q_{3n-1}}-\frac{a_{1}}{Q_{1}}+\frac{a_{1}a_{2}}{Q_{1}Q_{2}%
}=\sum_{k=1}^{n-1}\left(  -1\right)  ^{k-1}a_{1}a_{2}\cdots a_{3k}\left(
\frac{b_{3k+1}}{Q_{3k-1}Q_{3k+1}}+\frac{a_{3k+1}a_{3k+2}}{Q_{3k+1}Q_{3k+2}%
}\right)
\]%
\begin{align*}
&  =\sum_{k=1}^{n-1}\left(  -1\right)  ^{k-1}a_{1}a_{2}\cdots a_{3k}%
\frac{b_{3k+1}Q_{3k+2}+a_{3k+1}a_{3k+2}Q_{3k-1}}{Q_{3k-1}Q_{3k+1}Q_{3k+2}}\\
&  =\sum_{k=1}^{n-1}\left(  -1\right)  ^{k-1}a_{1}a_{2}\cdots a_{3k}%
\frac{b_{3k+1}b_{3k+2}Q_{3k+1}+a_{3k+2}\left(  b_{3k+1}Q_{3k}+a_{3k+1}%
Q_{3k-1}\right)  }{Q_{3k-1}Q_{3k+1}Q_{3k+2}},
\end{align*}
which proves Lemma \ref{LemTransfVarona}.
\end{proof}

\section{Proofs of theorems \ref{ThEuler} and \ref{ThHone}}

\label{sec:proofEandH}

The two proofs are similar, and consist in transforming the continued fraction%
\[
\frac{a_{1}}{b_{1}}%
\genfrac{}{}{0pt}{}{{}}{+}%
\frac{a_{2}}{b_{2}}%
\genfrac{}{}{0pt}{}{{}}{+\cdots}%
\genfrac{}{}{0pt}{}{{}}{+}%
\dfrac{a_{m}}{b_{m}}%
\]
by (\ref{SumEuler}) and (\ref{SumHone}), with $m=n$ and $m=2n$
respectively.\medskip

\noindent\textit{Proof of Theorem \ref{ThEuler}}. With the notations of
Section \ref{sec:notaion}, we prove first by induction that $Q_{k}=x_{k}$
$(k\geq0).$ Clearly $Q_{0}=1=x_{0}$ and $Q_{1}=x_{1}.$ Assuming that
$Q_{k}=x_{k}$ and $Q_{k+1}=x_{k+1},$ we obtain by (\ref{CFEuler1}) and
(\ref{Rec})
\begin{align*}
Q_{k+2}  &  =\left(  \theta x_{k+1}-\theta y_{k+1}\right)  x_{k+1}+\left(
\theta y_{k+1}\theta x_{k}\right)  x_{k}\\
&  =x_{k+2}-\left(  \theta y_{k+1}\right)  x_{k+1}+\left(  \theta
y_{k+1}\right)  x_{k+1}=x_{k+2},
\end{align*}
which proves that $Q_{k}=x_{k}$ $(k\geq0).$ Here $P_{1}=y_{1},$ $Q_{1}=x_{1},$
and
\[
\prod_{j=1}^{k+1}a_{j}=a_{1}\prod_{j=2}^{k+1}\theta y_{j-1}\theta
x_{j-2}=y_{1}\frac{y_{k+1}x_{k}}{y_{1}x_{0}}=x_{k}y_{k+1}.
\]
Since $Q_{k}=x_{k}$ $(k\geq0),$ we obtain by (\ref{SumEuler})%
\[
\frac{a_{1}}{b_{1}}%
\genfrac{}{}{0pt}{}{{}}{+}%
\frac{a_{2}}{b_{2}}%
\genfrac{}{}{0pt}{}{{}}{+\cdots}%
\genfrac{}{}{0pt}{}{{}}{+}%
\dfrac{a_{n}}{b_{n}}=\sum_{k=0}^{n-1}\left(  -1\right)  ^{k}\frac{y_{k+1}%
}{x_{k+1}},
\]
which proves Theorem \ref{ThEuler}.\medskip

\noindent\textit{Proof of Theorem \ref{ThHone}}. We prove by induction that%
\begin{equation}
Q_{2k}=x_{k},\quad Q_{2k+1}=\theta x_{k}-\theta y_{k}\quad\left(
k\geq0\right)  . \label{Hone5}%
\end{equation}
For $k=0,$ we have $Q_{0}=1=x_{0}$ and $Q_{1}=b_{1}=x_{1}-y_{1}=\theta
x_{0}-\theta y_{0}.$ Now assuming that it is true for some $k\geq0,$ we
compute%
\begin{align*}
Q_{2k+2}  &  =b_{2k+2}Q_{2k+1}+a_{2k+2}Q_{2k}=x_{k}\left(  \theta x_{k}-\theta
y_{k}\right)  +\theta y_{k}x_{k}=x_{k+1},\\
Q_{2k+3}  &  =b_{2k+3}Q_{2k+2}+a_{2k+3}Q_{2k+1}\\
&  =\frac{\theta^{2}x_{k}-\theta^{2}y_{k}}{x_{k}}x_{k+1}+\theta^{2}%
y_{k}\left(  \theta x_{k}-\theta y_{k}\right)  =\theta x_{k+1}-\theta y_{k+1}%
\end{align*}
by using (\ref{Rule}). Hence (\ref{Hone5}) is proved by induction. Now we
apply Lemma \ref{LemTransfHone}. First we have%
\[
a_{1}a_{2}\cdots a_{2k+1}=y_{k+1}\quad\left(  k\geq0\right)  .
\]
Indeed, this is clearly true for $k=0$ since $a_{1}=y_{1}$ and
\[
a_{1}a_{2}\cdots a_{2k+3}=a_{1}a_{2}\cdots a_{2k+1}a_{2k+2}a_{2k+3}%
=y_{k+1}\theta y_{k}\theta^{2}y_{k}=y_{k+1}\theta y_{k+1}=y_{k+2}.
\]
Replacing in (\ref{SumHone}) yields%
\[
\frac{a_{1}}{b_{1}}%
\genfrac{}{}{0pt}{}{{}}{+}%
\frac{a_{2}}{b_{2}}%
\genfrac{}{}{0pt}{}{{}}{+\cdots}%
\genfrac{}{}{0pt}{}{{}}{+}%
\dfrac{a_{2n}}{b_{2n}}=\sum_{k=0}^{n-1}\frac{y_{k+1}x_{k}}{x_{k+1}x_{k}}%
=\sum_{k=1}^{n}\frac{y_{k}}{x_{k}},
\]
which proves Theorem \ref{ThHone}.\medskip

\section{Proof of Theorem \ref{ThVarona}}

\label{sec:proofV}

It is simpler to prove first a slightly different result, namely

\begin{theorem}
\label{ThVarona1}For every integer $n\geq2,$%
\[
\sum_{k=1}^{n}\left(  -1\right)  ^{k-1}\frac{y_{k}}{x_{k}}=\frac{a_{1}}{b_{1}}%
\genfrac{}{}{0pt}{}{{}}{+}%
\frac{a_{2}}{b_{2}}%
\genfrac{}{}{0pt}{}{{}}{+\cdots}%
\genfrac{}{}{0pt}{}{{}}{+}%
\dfrac{a_{3n-4}}{b_{3n-4}},
\]
where
\begin{align*}
a_{1}  &  =y_{1}^{2},\quad a_{2}=x_{1}y_{2},\quad a_{3}=\frac{\theta y_{2}%
}{x_{1}},\\
b_{1}  &  =x_{1}y_{1},\quad b_{2}=\theta x_{1}-\theta y_{1},\quad b_{3}%
=\frac{\theta x_{2}}{x_{2}}-1,
\end{align*}
and for $k\geq1$%
\begin{align*}
a_{3k+1}  &  =1,\quad a_{3k+2}=y_{k+2},\quad a_{3k+3}=y_{k+1}\theta^{2}%
y_{k+1},\\
\quad b_{3k+1}  &  =1,\quad b_{3k+2}=x_{k+1}y_{k+1}-y_{k+2},\quad
b_{3k+3}=\frac{\theta^{2}x_{k+1}-\theta^{2}y_{k+1}}{x_{k+1}}-1.
\end{align*}

\end{theorem}

\begin{proof}
We prove by induction that, for every $k\geq1,$%
\begin{equation}
\left\{
\begin{array}
[c]{l}%
Q_{3k-1}=y_{1}y_{2}\cdots y_{k}x_{k+1},\\
Q_{3k}=y_{1}y_{2}\cdots y_{k}\left(  \theta x_{k+1}-x_{k+1}+\theta
y_{k+1}\right)  ,\\
Q_{3k+1}=y_{1}y_{2}\cdots y_{k}\left(  \theta x_{k+1}+\theta y_{k+1}\right)  .
\end{array}
\right.  \label{Varona9}%
\end{equation}
We have $Q_{0}=1$ and $Q_{1}=b_{1}=x_{1}y_{1}.$ Therefore%
\begin{align*}
Q_{2}  &  =b_{2}Q_{1}+a_{2}Q_{0}=\left(  \theta x_{1}-\theta y_{1}\right)
x_{1}y_{1}+x_{1}y_{2}=x_{2}y_{1},\\
Q_{3}  &  =b_{3}Q_{2}+a_{3}Q_{1}=\left(  \frac{\theta x_{2}}{x_{2}}-1\right)
x_{2}y_{1}+\frac{\theta y_{2}}{x_{1}}x_{1}y_{1}=y_{1}\left(  \theta
x_{2}-x_{2}+\theta y_{2}\right)  ,\\
Q_{4}  &  =b_{4}Q_{3}+a_{4}Q_{2}=Q_{3}+Q_{2}=y_{1}(\theta x_{2}+\theta y_{2}),
\end{align*}
which proves that (\ref{Varona9}) is true for $k=1.$ Now assuming that it is
true for some $k\geq1,$ we compute%
\begin{align*}
Q_{3k+2}  &  =b_{3k+2}Q_{3k+1}+a_{3k+2}Q_{3k}\\
&  =\left(  x_{k+1}y_{k+1}-y_{k+2}\right)  y_{1}\cdots y_{k}\left(  \theta
x_{k+1}+\theta y_{k+1}\right) \\
&  \qquad\qquad\qquad+y_{k+2}y_{1}\cdots y_{k}\left(  \theta x_{k+1}%
-x_{k+1}+\theta y_{k+1}\right) \\
&  =y_{1}\cdots y_{k}\left(  x_{k+1}y_{k+1}\theta x_{k+1}+x_{k+1}y_{k+1}\theta
y_{k+1}-x_{k+1}y_{k+2}\right) \\
&  =y_{1}\cdots y_{k+1}x_{k+2},\\
Q_{3k+3}  &  =b_{3k+3}Q_{3k+2}+a_{3k+3}Q_{3k+1}\\
&  =\left(  \frac{\theta^{2}x_{k+1}-\theta^{2}y_{k+1}}{x_{k+1}}-1\right)
y_{1}\cdots y_{k+1}x_{k+2}\\
&  \qquad\qquad\qquad+y_{k+1}\theta^{2}y_{k+1}y_{1}\cdots y_{k}\left(  \theta
x_{k+1}+\theta y_{k+1}\right) \\
&  =y_{1}\cdots y_{k+1}\left(  \theta^{2}x_{k+1}\theta x_{k+1}-x_{k+2}%
+\theta^{2}y_{k+1}\theta y_{k+1}\right) \\
&  =y_{1}\cdots y_{k+1}\left(  \theta x_{k+2}-x_{k+2}+\theta y_{k+2}\right)
,\\
Q_{3k+4}  &  =b_{3k+4}Q_{3k+3}+a_{3k+4}Q_{3k+2}=Q_{3k+3}+Q_{3k+2}\\
&  =y_{1}\cdots y_{k+1}\left(  \theta x_{k+2}+\theta y_{k+2}\right)  .
\end{align*}
Hence (\ref{Varona9}) is proved by induction. Now we apply Lemma
\ref{LemTransfVarona}. We have%
\begin{equation}
a_{1}a_{2}\cdots a_{3k}=y_{k+2}\left(  y_{1}y_{2}\cdots y_{k}\right)
^{2}\quad\left(  k\geq1\right)  . \label{Varona13}%
\end{equation}
Indeed, for $k=1$%
\[
a_{1}a_{2}a_{3}=y_{1}^{2}x_{1}y_{2}\frac{\theta y_{2}}{x_{1}}=y_{3}y_{1}^{2},
\]
and assuming that (\ref{Varona13}) holds for some $k\geq1,$
\[
a_{1}a_{2}\cdots a_{3k+3}=y_{k+2}\left(  y_{1}y_{2}\cdots y_{k}\right)
^{2}y_{k+2}y_{k+1}\theta^{2}y_{k+1}=y_{k+3}\left(  y_{1}y_{2}\cdots
y_{k+1}\right)  ^{2}.
\]
Using (\ref{Varona13}) in (\ref{SumVarona}), we obtain%
\begin{align*}
&  \frac{a_{1}}{b_{1}}%
\genfrac{}{}{0pt}{}{{}}{+}%
\frac{a_{2}}{b_{2}}%
\genfrac{}{}{0pt}{}{{}}{+\cdots}%
\genfrac{}{}{0pt}{}{{}}{+}%
\dfrac{a_{3n-1}}{b_{3n-1}}\\
&  =\frac{y_{1}}{x_{1}}-\frac{y_{1}x_{1}y_{2}}{x_{1}x_{2}y_{1}}+\sum
_{k=1}^{n-1}\left(  -1\right)  ^{k-1}y_{k+2}\left(  y_{1}\cdots y_{k}\right)
^{2}\frac{x_{k+1}y_{k+1}-y_{k+2}+y_{k+2}}{y_{1}\cdots y_{k}x_{k+1}y_{1}\cdots
y_{k+1}x_{k+2}}\\
&  =\frac{y_{1}}{x_{1}}-\frac{y_{2}}{x_{2}}+\sum_{k=1}^{n-1}\left(  -1\right)
^{k-1}\frac{y_{k+2}}{x_{k+2}}=\sum_{k=1}^{n+1}\left(  -1\right)  ^{k-1}%
\frac{y_{k}}{x_{k}},
\end{align*}
which proves Theorem \ref{ThVarona1}.
\end{proof}

Now, with the notations of Theorem \ref{ThVarona1}, we simply observe that%
\begin{align*}
\sum_{k=1}^{n}\left(  -1\right)  ^{k-1}\frac{y_{k}}{x_{k}}  &  =\frac{a_{1}%
}{b_{1}}%
\genfrac{}{}{0pt}{}{{}}{+}%
\frac{a_{2}}{b_{2}}%
\genfrac{}{}{0pt}{}{{}}{+}%
\frac{a_{3}}{b_{3}}%
\genfrac{}{}{0pt}{}{{}}{+}%
\frac{a_{4}}{b_{4}}%
\genfrac{}{}{0pt}{}{{}}{+}%
\frac{a_{5}}{b_{5}}%
\genfrac{}{}{0pt}{}{{}}{+\cdots}%
\genfrac{}{}{0pt}{}{{}}{+}%
\dfrac{a_{3n-4}}{b_{3n-4}}\\
&  =\frac{a_{1}}{b_{1}}%
\genfrac{}{}{0pt}{}{{}}{+}%
\frac{a_{2}}{b_{2}}%
\genfrac{}{}{0pt}{}{{}}{+}%
\frac{x_{1}a_{3}}{x_{1}b_{3}}%
\genfrac{}{}{0pt}{}{{}}{+}%
\frac{x_{1}a_{4}}{b_{4}}%
\genfrac{}{}{0pt}{}{{}}{+}%
\frac{a_{5}}{b_{5}}%
\genfrac{}{}{0pt}{}{{}}{+\cdots}%
\genfrac{}{}{0pt}{}{{}}{+}%
\dfrac{a_{3n-4}}{b_{3n-4}},
\end{align*}
which proves Theorem \ref{ThVarona}.

\section{Generalization of Hone expansions}

\label{sec:Hone}

In this section, we consider a sequence $F_{n}(X,Y)$ of nonzero polynomials
with non-negative integer coefficients and such that $F_{n}(0,0)=0$ for every
$n\geq0.$ Define the sequence $\left(  x_{n}\right)  _{n\geq0}$ by $x_{0}=1,$
$x_{1}\in\mathbb{Z}_{>0}$ and the recurrence relation%
\begin{equation}
x_{n+2}x_{n}=x_{n+1}^{2}\left(  F_{n}\left(  x_{n},x_{n+1}\right)  +1\right)
\quad\left(  n\geq0\right)  . \label{Rec1}%
\end{equation}
If $x_{n}$ satisfies (\ref{Rec1}), it is clear that%
\[
\theta^{2}x_{n}=F_{n}\left(  x_{n},x_{n+1}\right)  +1.
\]
It is easy to check by induction that $x_{n}$ is a positive integer and that
$x_{n}$ divides $x_{n+1}$ for every $n\geq0.$ Therefore%
\begin{equation}
x_{n+2}\geq x_{n+1}^{2}\frac{x_{n}+1}{x_{n}}>x_{n+1}^{2}\quad\left(
n\geq0\right)  . \label{Rec2}%
\end{equation}
Hence we deduce from (\ref{Rec2}) that $x_{2}\geq2$ and%
\begin{equation}
x_{n}\geq\left(  x_{2}\right)  ^{2^{n-2}}\geq2^{2^{n-2}}\quad\left(
n\geq2\right)  . \label{Min}%
\end{equation}
Now let $h$ be any positive integer. We define the series%
\[
S=\sum_{n=0}^{\infty}\frac{h^{n}}{x_{n+1}}=\frac{1}{h}\sum_{n=1}^{\infty}%
\frac{h^{n}}{x_{n}},
\]
which is convergent by (\ref{Min}). We can apply Theorem \ref{ThHone} above
with $y_{n}=h^{n},$ in which case $\theta y_{n}=h$ and $\theta^{2}y_{n}=1$ for
every $n\geq0.$ By letting $n\rightarrow\infty$ in Theorem \ref{ThHone}, we
get
\begin{equation}
S=\frac{a_{1}}{b_{1}}%
\genfrac{}{}{0pt}{}{{}}{+}%
\frac{a_{2}}{b_{2}}%
\genfrac{}{}{0pt}{}{{}}{+\cdots}%
\genfrac{}{}{0pt}{}{{}}{+}%
\dfrac{a_{n}}{b_{n}}%
\genfrac{}{}{0pt}{}{{}}{+\cdots}%
, \label{S}%
\end{equation}
where $a_{1}=1,$ $b_{1}=x_{1}-h,$ and for $k\geq1$%
\begin{align*}
a_{2k}  &  =h,\quad a_{2k+1}=1,\\
\quad b_{2k}  &  =x_{k-1},\quad b_{2k+1}=\frac{\theta^{2}x_{k-1}-1}{x_{k-1}%
}=\frac{F_{k-1}\left(  x_{k-1},x_{k}\right)  }{x_{k-1}}.
\end{align*}
Assume that $x_{1}>h.$ As $F_{n}(0,0)=0$ and $x_{n}$ divides $x_{n+1}$ for
every $n\geq0,$ we see that $a_{n}$ and $b_{n}$ are positive integers for
every $n\geq1$ in this case. If moreover $h=1$ and $F_{n}(x_{n},x_{n+1}%
)+1=F(x_{n+1})$ for some $F(X) \in{\mathbb{Z}}_{\geq0}[X]$, then (\ref{S})
gives the expansion in regular continued fraction of $S,$ already obtained by
Hone in \cite{Hone}.

\begin{example}
\label{ex:1} \label{ExHone1}The simplest of all sequences $(x_{n})$ satisfy
the recurrence relation%
\begin{equation}
x_{n+2}x_{n}=x_{n+1}^{2}\left(  x_{n}+1\right)  \quad\left(  n\geq0\right)  ,
\label{Simplest}%
\end{equation}
which means that $F_{n}(X,Y)=X$ for every $n\geq0.$ Let $h$ be a positive
integer, and assume that $x_{1}>h.$ Then we can apply the above results and we
get%
\[
S=\frac{1}{x_{1}-h}%
\genfrac{}{}{0pt}{}{{}}{+}%
\frac{h}{1}%
\genfrac{}{}{0pt}{}{{}}{+}%
\frac{1}{1}%
\genfrac{}{}{0pt}{}{{}}{+}%
\frac{h}{x_{1}}%
\genfrac{}{}{0pt}{}{{}}{+}%
\frac{1}{1}%
\genfrac{}{}{0pt}{}{{}}{+}%
\frac{h}{x_{2}}%
\genfrac{}{}{0pt}{}{{}}{+\cdots}%
\genfrac{}{}{0pt}{}{{}}{+}%
\dfrac{1}{1}%
\genfrac{}{}{0pt}{}{{}}{+}%
\dfrac{h}{x_{k}}%
\genfrac{}{}{0pt}{}{{}}{+\cdots}%
.
\]
In the case where $x_{1}=1$ and $h=1,$ we can apply this result starting with
$x_{2}=2$ in place of $x_{1}$ and we get%
\[
S-1=\frac{1}{1}%
\genfrac{}{}{0pt}{}{{}}{+}%
\frac{1}{x_{1}}%
\genfrac{}{}{0pt}{}{{}}{+}%
\frac{1}{1}%
\genfrac{}{}{0pt}{}{{}}{+}%
\frac{1}{x_{2}}%
\genfrac{}{}{0pt}{}{{}}{+\cdots}%
\genfrac{}{}{0pt}{}{{}}{+}%
\dfrac{1}{1}%
\genfrac{}{}{0pt}{}{{}}{+}%
\dfrac{1}{x_{k}}%
\genfrac{}{}{0pt}{}{{}}{+\cdots}%
.
\]
Hence, assuming that $x_{0}=x_{1}=1$ and $x_{n}$ satisfies (\ref{Simplest}),
we have%
\[
\left[  1;1,x_{1},1,x_{2},1,x_{3},1,x_{4},\ldots,1,x_{k},\ldots\right]
=\sum_{n=1}^{\infty}\frac{1}{x_{n}},
\]
which is (\ref{Nouv1}).
\end{example}

If the condition $x_{1}>h$ is not realized, let $N\geq0$ such that
$x_{N+1}>h.$ Then there exist a positive integer $t$ such that%
\[
S=\frac{t}{x_{N}}+h^{N}\sum_{n=0}^{\infty}\frac{h^{n}}{x_{n+N+1}}=\frac
{t}{x_{N}}+h^{N}\sum_{n=0}^{\infty}\frac{h^{n}}{x_{n+1}^{\prime}}=\frac
{t}{x_{N}}+h^{N}S^{\prime},
\]
where $x_{n}^{\prime}=x_{n+N}$ satisfies $x_{1}^{\prime}>h$ and the recurrence
relation%
\[
x_{n+2}^{\prime}x_{n}^{\prime}=\left(  x_{n+1}^{\prime}\right)  ^{2}\left(
F_{n+N}\left(  x_{n}^{\prime},x_{n+1}^{\prime}\right)  +1\right)  \quad\left(
n\geq0\right)  .
\]
Hence we can apply the above result to $S^{\prime}$ and we get%
\[
S=\frac{t}{x_{N}}+\frac{h^{N}a_{1}^{\prime}}{b_{1}^{\prime}}%
\genfrac{}{}{0pt}{}{{}}{+}%
\frac{a_{2}^{\prime}}{b_{2}^{\prime}}%
\genfrac{}{}{0pt}{}{{}}{+\cdots}%
\genfrac{}{}{0pt}{}{{}}{+}%
\dfrac{a_{n}^{\prime}}{b_{n}^{\prime}}%
\genfrac{}{}{0pt}{}{{}}{+\cdots}%
.
\]
This proves that%
\begin{equation}
\frac{1}{S}=\frac{x_{N}}{t}%
\genfrac{}{}{0pt}{}{{}}{+}%
\frac{h^{N}a_{1}^{\prime}x_{N}}{b_{1}^{\prime}}%
\genfrac{}{}{0pt}{}{{}}{+}%
\frac{a_{2}^{\prime}}{b_{2}^{\prime}}%
\genfrac{}{}{0pt}{}{{}}{+\cdots}%
\genfrac{}{}{0pt}{}{{}}{+}%
\dfrac{a_{n}^{\prime}}{b_{n}^{\prime}}%
\genfrac{}{}{0pt}{}{{}}{+\cdots}%
, \label{1/S}%
\end{equation}
which gives an expansion of $S^{-1}$ in a continued fraction whose terms are
positive integers.

\begin{example}
\label{ex:2} \label{ExHone2}Assume again that $x_{n}$ satisfies
(\ref{Simplest}), and take for example $x_{1}=1$ and $h=3.$ We have here%
\[
S=\sum_{n=0}^{\infty}\frac{3^{n}}{x_{n+1}}.
\]
with $x_{1}=1,$ $x_{2}=2,$ $x_{3}=8.$ Hence $N=2$ and $a_{1}^{\prime}=1,$
$b_{1}^{\prime}=5,$ and for $k\geq1$%
\[
a_{2k}^{\prime}=3,\quad a_{2k+1}^{\prime}=1,\quad b_{2k}^{\prime}%
=x_{k+2},\quad b_{2k+1}^{\prime}=1.
\]
By applying (\ref{1/S}), we obtain%
\[
\frac{1}{S}=\frac{2}{5}%
\genfrac{}{}{0pt}{}{{}}{+}%
\frac{18}{5}%
\genfrac{}{}{0pt}{}{{}}{+}%
\frac{3}{x_{3}}%
\genfrac{}{}{0pt}{}{{}}{+}%
\frac{1}{1}%
\genfrac{}{}{0pt}{}{{}}{+}%
\dfrac{3}{x_{4}}%
\genfrac{}{}{0pt}{}{{}}{+}%
\dfrac{1}{1}%
\genfrac{}{}{0pt}{}{{}}{+\cdots}%
\genfrac{}{}{0pt}{}{{}}{+}%
\dfrac{3}{x_{k}}%
\genfrac{}{}{0pt}{}{{}}{+}%
\dfrac{1}{1}%
\genfrac{}{}{0pt}{}{{}}{+\cdots}%
.
\]

\end{example}

\section{Generalization of Varona expansions}

\label{sec:Varona}

With the notations of Section \ref{sec:Hone}, we define now the series%
\[
T=\sum_{n=0}^{\infty}\left(  -1\right)  ^{n}\frac{h^{n}}{x_{n+1}}=\sum
_{n=1}^{\infty}\left(  -1\right)  ^{n-1}\frac{h^{n-1}}{x_{n}}.
\]
Here we have $y_{n}=h^{n-1}.$ By letting $n\rightarrow\infty$ in Theorem
\ref{ThVarona}, we get
\begin{equation}
T=\frac{a_{1}}{b_{1}}%
\genfrac{}{}{0pt}{}{{}}{+}%
\frac{a_{2}}{b_{2}}%
\genfrac{}{}{0pt}{}{{}}{+\cdots}%
\genfrac{}{}{0pt}{}{{}}{+}%
\dfrac{a_{n}}{b_{n}}%
\genfrac{}{}{0pt}{}{{}}{+\cdots}%
, \label{T}%
\end{equation}
where%
\begin{align*}
a_{1}  &  =1,\quad a_{2}=hx_{1},\quad a_{3}=h,\quad a_{4}=x_{1},\\
b_{1}  &  =x_{1},\quad b_{2}=\frac{x_{2}}{x_{1}}-h,\quad b_{3}=F_{1}\left(
x_{1},x_{2}\right)  +1-x_{1},\quad b_{4}=1,
\end{align*}
and for $k\geq2$%
\begin{align*}
a_{3k-1}  &  =h^{k},\quad a_{3k}=h^{k-1},\quad a_{3k+1}=1,\\
\quad b_{3k-1}  &  =h^{k-1}\left(  x_{k}-h\right)  ,\quad b_{3k}=\frac
{F_{k}\left(  x_{k},x_{k+1}\right)  }{x_{k}}-1,\quad b_{3k+1}=1.
\end{align*}
Assume that $x_{1}\geq h$ and that $F_{k}(X,Y)\neq X$ for every $k\geq0.$ Then
$b_{2}>0$ since $x_{2}>x_{1}^{2}\geq hx_{1}$ and all the $a_{n}$ and $b_{n}$
are positive integers. If moreover $h=1$ and $x_{1}=1$ we obtain the expansion
in regular continued fraction of $T$ given by Varona in \cite{Varona}.

\begin{example}
\label{ex:3} \label{ExVarona1}Assume that $(x_{n})$ satisfies (\ref{Simplest}%
). Then we cannot apply directly the above results since $F_{k}(X,Y)=X$ for
every $k\geq0$ and therefore $b_{3k}=0$ for every $k\geq1.$ However, by the
concatenation formula we have for $k\geq2$%
\begin{align*}
\frac{a_{3k-1}}{b_{3k-1}}%
\genfrac{}{}{0pt}{}{{}}{+}%
\frac{a_{3k}}{0}%
\genfrac{}{}{0pt}{}{{}}{+}%
\dfrac{a_{3k+1}}{b_{3k+1}}%
\genfrac{}{}{0pt}{}{{}}{+}%
\frac{A}{B}  &  =\frac{a_{3k-1}}{b_{3k-1}+\dfrac{a_{3k}}{a_{3k+1}}b_{3k+1}}%
\genfrac{}{}{0pt}{}{{}}{+}%
\frac{a_{3k}A}{a_{3k+1}B}\\
&  =\frac{h}{x_{k}-h+1}%
\genfrac{}{}{0pt}{}{{}}{+}%
\frac{A}{B}.
\end{align*}
Then we have for $n\geq3$%
\begin{align*}
\frac{a_{1}}{b_{1}}%
\genfrac{}{}{0pt}{}{{}}{+}%
&  \frac{a_{2}}{b_{2}}%
\genfrac{}{}{0pt}{}{{}}{+}%
\frac{a_{3}}{b_{3}}%
\genfrac{}{}{0pt}{}{{}}{+}%
\dfrac{a_{4}}{b_{4}}%
\genfrac{}{}{0pt}{}{{}}{+}%
\frac{a_{5}}{b_{5}}%
\genfrac{}{}{0pt}{}{{}}{+\cdots}%
\genfrac{}{}{0pt}{}{{}}{+}%
\dfrac{a_{3n-5}}{b_{3n-5}}%
\genfrac{}{}{0pt}{}{{}}{+}%
\frac{a_{3n-4}}{b_{3n-4}}\\
&  =\frac{1}{x_{1}}%
\genfrac{}{}{0pt}{}{{}}{+}%
\frac{hx_{1}}{x_{1}^{-1}x_{2}-h}%
\genfrac{}{}{0pt}{}{{}}{+}%
\frac{h}{1}%
\genfrac{}{}{0pt}{}{{}}{+}%
\dfrac{x_{1}}{1}%
\genfrac{}{}{0pt}{}{{}}{+}%
\frac{h}{x_{2}-h+1}%
\genfrac{}{}{0pt}{}{{}}{+\cdots}%
\genfrac{}{}{0pt}{}{{}}{+}%
\frac{h}{x_{n-2}-h+1}%
\genfrac{}{}{0pt}{}{{}}{+}%
\frac{h}{x_{n-1}-h}.
\end{align*}
In the case where $x_{1}=1$ and $h=1,$ we get%
\[
\sum_{k=1}^{n}\frac{\left(  -1\right)  ^{k-1}}{x_{k}}=\frac{1}{1}%
\genfrac{}{}{0pt}{}{{}}{+}%
\frac{1}{1}%
\genfrac{}{}{0pt}{}{{}}{+}%
\frac{1}{1}%
\genfrac{}{}{0pt}{}{{}}{+}%
\dfrac{1}{1}%
\genfrac{}{}{0pt}{}{{}}{+}%
\frac{1}{x_{2}}%
\genfrac{}{}{0pt}{}{{}}{+\cdots}%
\genfrac{}{}{0pt}{}{{}}{+}%
\frac{1}{x_{n-2}}%
\genfrac{}{}{0pt}{}{{}}{+}%
\frac{1}{x_{n-1}-1}%
\]
for $n\geq3.$ This yields (\ref{Nouv2}) by letting $n\rightarrow\infty.$
\end{example}

\begin{remark}
Hone and Varona in \cite{Hone2} and \cite{Hone3} have recently generalized
their results to sums of a rational number and certain Engel or Pierce series
by giving their expansions in regular continued fractions.
\end{remark}

\section*{Acknowledgement}

The last named author was supported by the Research Institute for Mathematical
Sciences, an International Joint Usage / Research Center located in Kyoto University.


\begin{thebibliography}{9}                                                                                                %


\bibitem {Euler}L. Euler, Introductio in analysin infinitorum, Lausanne, 1748.

\bibitem {Hone}A. N. W. Hone, Curious continued fractions, non linear
recurrences and transcendental numbers, J. Integer Seq. 18 (8) (2015), article
15.8.4, 10 pp.

\bibitem {Hone2}A. N. W. Hone, Continued fractions for some transcendental
numbers, Monatsh. Math. 182 (2017), 33--38.

\bibitem {Hone3}A. N. W. Hone and J. L. Varona, Continued fractions and
irrationality exponents for modified Engel and Pierce series, Monatsh. Math.
190 (2019), 501--516.

\bibitem {Varona}J. L. Varona, The continued fraction expansion of certain
Pierce series, J. Number Theory 180 (2017), 573--578.
\end{thebibliography}
\end{document}